\documentclass [11pt]{amsart}
\usepackage{amsthm, amsmath, amssymb,latexsym}
\usepackage[usenames,dvipsnames]{color}
\usepackage[all]{xy}

\newtheorem{theorem}{Theorem}[section]
\newtheorem{lemma}[theorem]{Lemma}
\newtheorem{corollary}[theorem]{Corollary}
\newtheorem{proposition}[theorem]{Proposition}
\newtheorem{remark}[theorem]{Remark}

\newcommand{\A}{\mathsf{A}}
\newcommand{\Ag}{\mathsf{A}_g}
\newcommand{\M}{\mathsf{M}}
\newcommand{\Tg}{\mathsf{T}_g}
\newcommand{\Mg}{\mathsf{M}_g}
 
\newcommand{\nc}{\newcommand} 
\nc{\cH}{{\mathcal H}}
\nc{\cA}{{\mathcal A}}
\nc{\cG}{{\mathcal G}}
\nc{\cC}{{\mathcal C}}
\nc{\cO}{{\mathcal O}}
\nc{\cI}{{\mathcal I}}
\nc{\cB}{{\mathcal B}}
\nc{\cY}{{\mathcal Y}}
\nc{\cK}{{\mathcal K}} 
\nc{\cX}{{\mathcal X}}
\nc{\cS}{{\mathcal S}}
\nc{\cE}{{\mathcal E}}
\nc{\cF}{{\mathcal F}}
\nc{\cZ}{{\mathcal Z}}
\nc{\cQ}{{\mathcal Q}}
\nc{\cN}{{\mathcal N}}
\nc{\cP}{{\mathcal P}}
\nc{\cL}{{\mathcal L}}
\nc{\cM}{{\mathcal M}}
\nc{\cT}{{\mathcal T}}
\nc{\cW}{{\mathcal W}}
\nc{\cU}{{\mathcal U}}
\nc{\cJ}{{\mathcal J}}
\nc{\cV}{{\mathcal V}}
\nc{\bH}{{\mathbb H}}
\nc{\bA}{{\mathbb A}}
\nc{\bG}{{\mathbb G}}
\nc{\bC}{{\mathbb C}}
\nc{\bO}{{\mathbb O}}
\nc{\bI}{{\mathbb I}}
\nc{\bB}{{\mathbb B}}
\nc{\bY}{{\mathbb Y}}
\nc{\bK}{{\mathbb K}} 
\nc{\bX}{{\mathbb X}}
\nc{\bS}{{\mathbb S}}
\nc{\bE}{{\mathbb E}}
\nc{\bF}{{\mathbb F}}
\nc{\bZ}{{\mathbb Z}}
\nc{\bQ}{{\mathbb Q}}
\nc{\bN}{{\mathbb N}}
\nc{\bP}{{\mathbb P}}
\nc{\bL}{{\mathbb L}}
\nc{\bM}{{\mathbb M}}
\nc{\bT}{{\mathbb T}}
\nc{\bW}{{\mathbb W}}
\nc{\bU}{{\mathbb U}}
\nc{\bD}{{\mathbb D}}
\nc{\bJ}{{\mathbb J}}
\nc{\bV}{{\mathbb V}}
\nc{\bbZ}{{\mathbb Z}}
\nc{\bR}{{\mathbb R}}
\nc{\fr}{{\rightarrow}}
\nc{\co}{{\nabla}}

\newcommand{\la}{\longrightarrow}
\nc{\cu}{{\barline{\nabla}}}

\newcommand{\ra}{{\rightarrow}}

\begin{document}
\author{Paola Frediani, Gian Pietro Pirola}
\title {On the geometry of the second fundamental form of the Torelli map} 

\address{Universit\`{a} di Pavia} \email{paola.frediani@unipv.it}
\email{gianpietro.pirola@unipv.it} 

\thanks{The authors are members of GNSAGA of INdAM.
The authors were partially supported by national MIUR funds,
PRIN  2017 Moduli and Lie theory,  and by MIUR: Dipartimenti di Eccellenza Program
   (2018-2022) - Dept. of Math. Univ. of Pavia.
    } \subjclass[2000]{14H10;14H15;14H40;32G20}



\maketitle

   \date{}                         



\begin{abstract}

In this paper we give a geometric interpretation of the second fundamental form of the period map of curves and we use it to improve the  upper bounds on the dimension of a totally geodesic subvariety $Y$ of $\A_g$ generically contained in the Torelli locus obtained in \cite{cfg}, \cite{gpt}. We get $\dim Y \leq 2g-1$ if $g$ is even, $\dim Y \leq 2g$ if $g$ is odd. We also study totally geodesic subvarieties $Z$ of $\A_g$  generically contained in the hyperelliptic Torelli locus and we show that $\dim Z \leq g+1$. 
\end{abstract}

\section{Introduction}
The aim of this paper is to study the local geometry of the immersion given by the period map of curves. 
Let $\M_g$ be the moduli space of smooth complex projective curves
of genus $g$,  $\Ag$ the moduli space of principally polarized
abelian varieties of dimension $g$. Denote by $j : \M_g \fr \A_g$ the period map or Torelli map and by $\Tg$ the Torelli locus, that is the closure of $j(\Mg)$ in $\Ag$. We consider $ \A_g$ endowed with the (orbifold) metric induced by the symmetric metric on the Siegel space 
${\mathsf H}_g$ of which $ \A_g$  is a quotient by the symplectic group $Sp(2g, \bZ)$. We want to relate the Torelli locus  to the geometry of  $ \A_g$ considered as a locally symmetric variety (see also \cite{moonen-oort,cf1,cfg,deba} for motivation and related problems).  More precisely, we are interested in studying totally geodesic subvarieties $Y$ of $ \A_g$ which are generically contained in $\Tg$, i.e. such that $Y$ is contained in  $\Tg$ and $Y \cap j(\M_g) \neq \emptyset$.    
A totally geodesic subvariety $Y$ of $ \A_g$ is an algebraic subvariety which is the image of a totally geodesic submanifold $\tilde{Y} \subset  {\mathsf H}_g$. One expects that there exist very few totally geodesic subvarieties of $ \A_g$ generically contained in $\Tg$, at least for high genus $g$.  This is related with a conjecture of Coleman and Oort according to which, for $g$ sufficiently high, there should not exist Shimura subvarieties of $ \A_g$ generically contained in $\Tg$ (\cite{moonen-oort}.). In fact Shimura subvarieties of $ \A_g$ are those totally geodesic subvarieties of $ \A_g$ admitting a point with complex multiplication (\cite{moonen-JAG}). 

To study totally geodesic subvarieties of $ \A_g$ generically contained in the Torelli locus, we  compute the second fundamental form of the Torelli map. This also allows to study totally geodesic subvarieties that are not necessarily algebraic. The computation of the second fundamental form of the Torelli map  was initiated in \cite{cpt} and continued in \cite{cfg} and \cite{gpt}. In \cite{cfg} a bound for the maximal dimension of a germ of a totally geodesic subvariety $Y$ of  $ \A_g$ contained in $\Tg$ is given in terms of the gonality of a point $[C] \in  \M_g$ such that $j([C] ) \in Y$. From this one obtains a bound for the  the maximal dimension of a germ of a totally geodesic subvariety $Y$ of  $ \A_g$ contained in $\Tg$ for $g \geq 4$, which only depends on the genus: $\dim Y \leq \frac{5}{2}(g-1)$.  This bound was recently improved in \cite{gpt}, where it is proven that $\dim Y \leq (7g-2)/3$, using the Fujita decomposition of the Hodge bundle of a (real one-dimensional) family of abelian varieties given by a geodesic in ${\mathsf H}_g$ and results on the Massey products obtained in \cite{pt}, \cite{gst}.  

In this paper we further improve such a bound by developing the techniques  used in \cite{cfg}. 

We give a geometric interpretation of  one of the main results of \cite{cfg} where the second fundamental form was expressed in terms of a multiplication by an intrinsic double differential of the second kind on the product $C\times C$.  This form appears also in an unpublished book of Gunning \cite{gun}.

In this way we are able  to use a family of quadrics of rank 4 containing the canonical curve whose images via the second fundamental form give quadrics that can be simultaneously diagonalised on a suitable linear subspace of $H^1(T_C)$. This allows to give a better estimate on the rank of the second fundamental form. 

The main results are the following 
\begin{theorem}
\label{gonality}
If $C$ is a smooth curve of genus $g>5$, gonality $k$, it has no involutions and is not a smooth plane curve,  then any totally geodesic subvariety generically contained in the Torelli locus and passing through $j(C)$ has
dimension $$m\leq  \frac {3(g-1)}{2}+k.$$
\end{theorem}

From this, recalling that the gonality $k\leq [(g+3)/2]$, we obtain the following estimate (the cases $g=4,5$ follow from \cite[Thm.4.4]{cfg}). 

\begin{theorem}
Let $Y$ be a germ of a totally geodesic submanifold generically contained in the Torelli locus $\Tg$, $g>5$,  then  $\dim Y \leq 2g-1$ if $g$ is even, $\dim Y \leq 2g$ if $g$ is odd.
\end{theorem}

These techniques are also used to give a good estimate on the dimension of totally geodesic  subvarieties $Y$ of  $ \A_g$ generically contained in the hyperelliptic Torelli locus $\mathsf {TH}_g$, i.e. contained in the closure of the image of the hyperelliptic locus $j(\mathsf {HE}_g)$ and such that $Y \cap j(\mathsf {HE}_g) \neq \emptyset$. 

The analogous of the Coleman Oort conjecture for Shimura subvarieties generically contained in  the hyperelliptic Torelli locus was considered by Lu and Zuo in \cite{lu-zuo} where they proved that  for $g>7$ there do not exist Kuga curves generically contained in the hyperelliptic Torelli locus. In particular for $g>7$ there do not exist Shimura curves generically contained in the hyperelliptic Torelli locus. 

Our  techniques are different, since they are local, hence allow to study totally geodesic subvarieties which are not necessarily algebraic and of any dimension.

In fact recall that the Torelli map is an immersion outside the hyperelliptic locus but also if restricted to the hyperelliptic locus (\cite{os}), so it is possible to study the second fundamental form of the restriction of the Torelli map to the hyperelliptic locus (see also \cite{cf1}). 

We prove the following 

\begin{theorem}
Let $Y$ be a germ of a totally geodesic submanifold of ${\mathsf A}_g$ contained in $ \mathsf {TH}_g$ then $\dim Y \leq g+1$. 
\end{theorem}

 For $g=3$ we get a better bound, namely 
 \begin{proposition}
Let $Y$ be a germ of a totally geodesic submanifold of ${\mathsf A}_3$,  contained in the hyperelliptic Torelli locus, then $\dim Y \leq 3$. 
\end{proposition}
 
Notice that for low genus there are examples of Shimura (hence totally geodesic) subvarieties of  ${\mathsf A}_g$ contained in  $ \mathsf {TH}_g$, namely the ones given by families (8), (22), (36), (39) of Table 2 of \cite{fgp} (see also Tables 1 of \cite{moo} for family (8), Table 2 of \cite{moonen-oort} for family (22)).  Family (8) yields a 2-dimensional Shimura subvariety generically contained in  $\mathsf {TH}_3$, while families (22), (36), (39) yield one dimensional Shimura subvarieties  generically contained respectively in  $\mathsf {TH}_3$,  $\mathsf {TH}_4$,  $\mathsf {TH}_5$. \\

{\bfseries \noindent{Acknowledgements.} }  
The second author would like to thank Indranil Biswas for the very
interesting conversations during his visit at Tata Institute of
fundamental research, Mumbai in April 2019. 

\section{Second fundamental form}
In this section we recall the definition and the results on the second fundamental form of the Torelli locus obtained in \cite{cpt} and in \cite{cfg} and we give a geometric interpretation of the form $\hat{\eta}$ introduced in \cite{cfg} .

Let $\mathsf{M}_g$ be the moduli space of smooth complex algebraic curves of genus $g$, and let $\mathsf{A}_g$ be the moduli space of principally polarised abelian varieties of dimension $g$. Denote by $j: \mathsf{M}_g \rightarrow  \mathsf{A}_g$ the Torelli map. Both $\mathsf{M}_g$ and $\mathsf{A}_g$ are complex orbifolds. 
The space  $\mathsf{A}_g$ is the quotient of the Siegel space ${\mathsf H}_g$ by the action of the  symplectic group $Sp(2g, \bZ)$, hence  $\mathsf{A}_g$ is endowed with the orbifold locally symmetric metric (called the Siegel metric) which is induced by the symmetric metric on the Siegel space ${\mathsf H}_g$. 
We denote by $\nabla$ the corresponding Levi Civita connection. Recall that the Torelli map is an orbifold immersion outside the hyperelliptic locus and also restricted to the hyperelliptic locus (\cite{os}). 

Assume now that $g \geq 4$. 
Outside the hyperelliptic locus  we have the following exact sequence of tangent bundles associated to the orbifold immersion  $j: \mathsf{M}_g \rightarrow  \mathsf{A}_g$ evaluated at a non hyperelliptic curve $[C]$:
\begin{equation}\label{short_exact_intro}
0 \ra T_{[C]} \mathsf{M}_g \xrightarrow{dj} T_{([JC],\Theta)} \mathsf{A}_g \xrightarrow{\pi} N_{[JC],\Theta} \ra 0. 
\end{equation}
Consider its dual 
\begin{equation}
0 \ra I_2(K_C) \ra S^2H^0(K_C) \ra H^0(2K_C) \ra 0.
\end{equation}

Denote by 
\begin{equation}
\label{II}
II: S^2 T_{[C]} \mathsf{M}_g \longrightarrow  N_{([JC],\Theta)}
\end{equation}
 the second fundamental form of the Torelli map, and by 
 \begin{equation}
 \label{rho}
 \rho:= II^{\vee} : I_2(K_C) \ra S^2H^0(C,2K_C)\cong S^2H^1(C,T_C)^\vee
 \end{equation}
 its dual.

We shall state a result obtained in \cite{cpt} which gives an expression of $\rho(Q)(v \odot w)$, where $Q \in I_2(K_C)$ and  $v, w \in H^1(C, T_C)$ are Schiffer variations. First let us recall the definition of a Schiffer variation at a point $p \in C$.

Consider the exact sequence 
\begin{equation}
0 \ra T_C \ra T_C(p) \ra T_C(p)_{|p} \ra 0.
\end{equation}
The coboundary map gives an injection $\delta: H^0(C,T_C(p)_{|p} )  \hookrightarrow H^1(C, T_C)$. A Schiffer variation at $p$ is a generator of $\delta(H^0(C,T_C(p)_{|p} ) $. 
Choose a local coordinate $(U,z)$ at $p$ and $b$ a bump function which is equal to 1 in a neighbourhood of $p$. Then the form  $\theta:= \frac{\bar{\partial}b}{z} \cdot \frac{\partial}{\partial z}\in A^{0,1}(T_C)$ is a Dolbeault representative of a Schiffer variation at $p$. One can easily check that the map 
$$\xi: T_pC \ra H^1(C, T_C), \ u:= \lambda \frac{\partial}{\partial z}(p) \mapsto \xi_u:= \lambda^2 [\theta]$$
is independent of the choice of the local coordinate $z$.

Consider a curve $C$ of genus $g \geq 2$, and take a point $p \in C$. The space $H^0(C,K_C(2p))$ of meromorphic $1$-forms on $C$ with a double pole on $p$ injects into $H^1(C \setminus \{p\}, \mathbb{C}) \cong H^1(C, \mathbb{C})$.  In fact, if a meromorphic one form $\omega \in H^0(C,K_C(2p))$  were exact, there would exist a meromorphic  function $f$ on $C$ with a simple pole at $p$ and holomorphic elsewhere such that $\omega = df$, hence $C$ would be isomorphic to ${\mathbb P}^1$. 
 Denote by $j_p: \; H^0(C,K_C(2p)) \hookrightarrow H^1(C,\mathbb{C})$ this injection. 

Observe that $h^0(C,K_C(2p))=g+1$ and  $H^0(C, K_C) \subset H^0(C, K_C(2p))$ is mapped by $j_p$ onto $H^{1,0}(C)$, hence the preimage  $j_p^{-1}(H^{0,1}(C))$ has dimension $1$. Fix a local coordinate $(U,z)$ centered in $p$. Then there exists a unique element $\phi$ in this line whose expression on $U\setminus \{p\}$ is
\begin{equation}
\phi := \bigg( \frac{1}{z^2}+h(z) \bigg) dz,
\end{equation}
where $h$ is a holomorphic function. The form $\eta_p$ is defined as follows:
\begin{equation}\begin{split}
\eta_p: T_pC &\longrightarrow H^0(C,K_C(2p)), \\
u=\lambda \frac{\partial}{\partial z}(p) &\longmapsto \eta_p(u)=\lambda \phi.
\end{split}\end{equation}
One can easily prove that $\eta_p$ is independent of the choice of the local coordinate. 

Denote by $\mu_2: I_2(K_C) \ra H^0(4K_C)$ the second gaussian map of the canonical bundle (\cite{wh}). 

Let us now state a result of \cite{cpt}. 

\begin{theorem}[Colombo, Pirola, Tortora  ~\cite{cpt}]
\label{copito}
Let $C$ be a non-hyperelliptic curve of genus $g \geq 4$. Let $p,q \in C$ and $u \in T_p C$, $v \in T_q C$. Then:
\begin{equation}\label{rhomu_equazione}\begin{split}
&\rho (Q) (\xi _u \odot \xi _v) = -4 \pi i \eta _{p} (u)(v)Q(u,v), \\
&\rho(Q) (\xi_u \odot \xi_u) = -2 \pi i \mu _2 (Q) (u^{\otimes 4}).
\end{split}\end{equation}
\end{theorem}


 Theorem \ref{copito} is used  in \cite{cf1} to compute the holomorphic sectional curvature of  $\mathsf{M}_g$ with respect to the Siegel metric along the Schiffer variations.

 Let us now recall a more intrinsic description of the form $\eta_p$  obtained in \cite{cfg}.

  Let $S:= C\times C$, let $\Delta$ denote the diagonal and let $\pi_1, \pi_2:
  S \ra C$ be the projections ${\pi_1}(x,y) =x$, $\pi_2(x,y) = y$.  Then $ K_S =
  \pi_1^*K_C \otimes \pi_2^*K_C$. Consider the line bundle $L:= K_S(2\Delta)$
  on $S$ and set  $ V:={\pi_1}_*(\pi_2^*K_C ( 2\Delta))$, $ E:={\pi_1}_* L \cong  K_C \otimes V$  by the projection formula. We have $H^0(
  \pi_1^{-1}(x) , \pi_2^*K_C (2\Delta)) \cong H^0 (C,K_C (2x)) $, hence $V$ is a holomorphic vector bundle on $C$
  with fibre $V_x \cong H^0( C , K_C ( 2x)) $ and the map $x\mapsto
  \eta_x$ is a section of $ E$ that we call $\eta$.
  Since $E={\pi_1}_* L$ there is an isomorphism $ H^0(C,E) \cong H^0(S,L)$ and $\eta$ corresponds to a global section $\hat{\eta} \in  H^0(S,L)$ such that  for  $u\in
  T_xC$ and $v\in T_y C$ with $x\neq y$, we have $ \eta_x
  (u)(v)=\hat{\eta} (u,v) $.

\begin{proposition}[\cite{cfg}]
 The section $\eta$  of $E$ is holomorphic.
Moreover the form $\hat{\eta} $ is symmetric, i.e. $\hat{\eta} (u,v) = \hat{\eta} (v,u)$.

\end{proposition}



The exact sequence on the surface $S$
$$0 \to K_S(-\Delta) \rightarrow K_S \rightarrow  {K_S}_{|\Delta} \rightarrow 0,$$
induces an exact sequence on  global sections

$$0 \to H^0( K_S(-\Delta) )\rightarrow H^0(K_S) \rightarrow  H^0({K_S}_{|\Delta}) \rightarrow 0.$$
With the natural identifications  $H^0(K_S) \cong H^0(K_C) \otimes H^0(K_C)$,  $H^0({K_S}_{|\Delta}) \cong H^0(2K_C)$, the restriction to the diagonal is identified with the multiplication map $H^0(K_C) \otimes H^0(K_C) \to  H^0(2K_C)$. Hence its kernel $H^0( K_S(-\Delta) )$ is isomorphic to $\Lambda^2(H^0(K_C)) \oplus I_2(K_C)$.  Since
  elements of $I_2(K_C)$ are symmetric, they are in fact contained in
  $H^0(S,K_S(-2\Delta))$. So if $Q\in I_2(K_C)$,  the section $Q\cdot
\hat{\eta} $ lies in $H^0(S,2K_S) \cong H^0(C, 2K_C) \otimes H^0(C,
  2K_C)$.
\begin{theorem}[\cite{cfg}]
  \label{prodotto} With the above identifications, if $C$ is
  non-hyperelliptic and of genus $g\geq 4$, then $\rho : I_2(K_C) \ra
  S^2H^0(C,2K_C)$ is the restriction to $I_2(K_C)$ of the multiplication
  map
  \begin{equation}
    H^0(S, K_S(-2\Delta)) \longrightarrow H^0(S,2K_S) \qquad Q \mapsto Q \cdot \hat{\eta} .
  \end{equation}
\end{theorem}

The form $\hat{\eta}$  is very mysterious, locally one has
\begin{equation}
\hat{\eta}= \frac{dz\wedge dw}{(z-w)^2}+ f(z,w)dz\wedge dw \label{2diff}\end{equation} in coordinates near the diagonal and $f(z,w)=f(w,z)$ is smooth. It appears also in an unpublished book of Gunning under the name of "intrinsic double differential of the second kind'' \cite{gun}. 
It seems very hard to give an actual computation of $\hat{\eta}.$ But we have the following geometric argument that gives information.

Take a quadric $Q\in I_2(K_C) \subset S^2H^1(\cO_C)^{\vee}$, $Q: H^1(\cO_C)\times H^1(\cO_C)\to \bC$. 

By the above discussion,  we see the quadric $Q \in I_2(K_C)$  as  a holomorphic section $\tilde{Q}$ of $K_S$ which belongs to $H^0( K_S(-2\Delta) )$, hence it vanishes on the diagonal.

Denote by $\phi_{K_C}: C \to \bP H^1(\cO_C)$ the canonical map. 
Take $\{\omega_1,...,\omega_g\}$ a basis of $H^0(K_C)$, then we have $Q = \sum_{i,j} a_{ij} \omega_i \odot \omega_j$, while $\tilde{Q} = \sum_{ij} a_{ij} \pi_1^*(\omega_i) \wedge \pi_2^*(\omega_j) $ as a section of $H^0(K_S)$.  
Take $ p, q \in C$,  then  we have the identification $T_{(p,q)}S  \cong \pi_1^*T_pC \oplus \pi_2^* T_qC$. So, if we take $u_p \in T_pC$, $u_q \in T_qC$, we have 
$$\tilde{Q}(u_p,u_q) = \sum_{i,j} a_{ij} \omega_i(u_p) \omega_j(u_q), $$which is equal to the value of the bilinear form $Q$ at the points  $\phi_{K_C}(p)$ and  $\phi_{K_C}(q)$.

Since $\tilde{Q}$ vanishes on the diagonal, 
the zero locus  $Z(Q)$ of $\tilde{Q}$  is 
$$\{(p,q) \in S : p\neq q : \tilde{Q}(u_p,u_q)=0\} \cup \Delta.$$
Since $\tilde{Q}(u_p,u_p)=\tilde{Q}(u_q,u_q)=0$,  if $(p,q) \in Z(Q)$, then the symmetric bilinear form $Q$ vanishes at the points $(\phi_{K_C}(p), \phi_{K_C}(q))$, $(\phi_{K_C}(p), \phi_{K_C}(p))$, and $(\phi_{K_C}(q), \phi_{K_C}(q))$.  Hence the line $L(p,q)$ in  $ \bP H^1(\cO_C)$  through $\phi_{K_C}(p)$ and  $\phi_{K_C}(q)$ is contained in $Q$, therefore
$$Z(Q)=\{(p,q) \in S :   Q(x,x)=0, \   \forall x\in L(p,q) \}.$$

Similarly if $Q'=\rho(Q) \in S^2H^0(2K_C) \subset H^0(2K_S)$, set
$$Z(Q')=\{(p,q) \in S  : Q'(p,q)=0\}.$$

By Theorem \ref{prodotto} we have  
$$Z(Q')=Z(Q)\cup Z(\hat{\eta}).$$
Then if $p \neq q$, $(p,q) \in Z(Q)$ implies that $(p,q) \in Z(Q')$,  but this time $Q'(p,p)\neq 0$ in general (the intersection with the diagonal is the divisor of the image of the second Gaussian map).
In particular,  if the quadric $Q$ has rank $4$,  we project form the kernel of $Q$ to $\bP^3$  to get a quadric  $Q_{\bP^3}$ of $\bP^3$ which is  isomorphic to $\bP^1\times \bP^1.$  Call $f: C \rightarrow \bP^1$ the composition of the canonical map with the projection form the kernel of $Q$ and finally with one of the two projections from $\bP^1 \times \bP^1$  to $\bP^1$. For $y\in \bP^1$,  set $f^{-1}(y)=\{P_1\dots P_d\}$. If $y$ is general we get for $i\neq j$, $\tilde{Q}(P_i,P_j)=Q'(P_i,P_j)=0$, but
$Q'(P_i,P_i)\neq 0$. Notice that here (and also later in the paper) we use the simplified notation $\tilde{Q}(P_i,P_j)$ instead of writing $\tilde{Q}(u_i,u_j)$, with $u_i \in T_{P_i}C$, $u_j  \in T_{P_j}C$. This is clearly possible once we choose a local trivialisation of the line bundle $K_S$ in a neighborhood of $(P_i, P_j)$. When there is no ambiguity,  we will also identify $Q$ and $\tilde{Q}$ as in Theorem \ref{prodotto}. 

\begin{theorem}
\label{interpretation}
If $Q$ has rank $4$ the quadric $Q'$ has rank $\geq \deg f$ where $f$ is  the composition of the canonical map with the projection to one of the two factors of $\bP^1\times \bP^1.$
\end{theorem}
\proof

Consider the linear subspace $V = \langle \xi_{P_1}, ..., \xi_{P_d} \rangle \subset H^1(T_C)$. It is easy to see that $\dim V = d$ (see e.g. the proof of Theorem 4.1 of \cite{cfg}). The quadric $Q'$ belongs to $S^2 H^1(T_C)^{\vee}$ and the image of the points $P_i$ via the bicanonical map are the points $[\xi_{P_i}] \in  \bP(H^1(T_C))$. Hence the above argument shows that the matrix associated to the restriction of $Q' = \rho(Q)$  to $V$ in the basis  given by the Schiffer variations at the points $P_i$ is diagonal with non zero diagonal entries. 
\qed\\

Recall that quadrics $Q \in I_2(K_C)$ of rank 4 correspond to $\{L, K_C(-L), U, W\}$ where $L$ is a line bundle on $C$, $U \subset H^0(C,L)$ and $W \subset H^0(K_C(-L))$ are 2-dimensional subspaces. The quadric $Q$ has rank 3 if and only if $2L = K_C$ and $U =W$. In the case $rank(Q) =4$,  if $U = \langle s_0,s_1\rangle $ and $W = \langle t_0, t_1\rangle$, 
$Q= s_0t_0\odot s_1t_1 - s_0t_1 \odot s_1 t_0. $
So, if we denote by $f,g:C \rightarrow \bP^1$ the two maps given by the two pencils $U$ and $W$, we see that the zero locus of $Q$ in $S$ is given by
$$\{(p,q) \in S \ | \  f(p) = f(q)\} \cup \{(p,q) \in S \ | \  g(p) = g(q)\} \cup ( (B_f \cup B_g) \times   C) \cup  (C \times (B_f \cup B_g)), $$
where $B_f$ and $B_g$ are the base loci of the pencils $f$ and $g$. 

If $rank(Q) =3$ and $U=W = \langle s_0,s_1\rangle$, then $Q = s_0^2 \odot s_1^2 - s_0s_1 \odot s_0s_1$. If we denote by $h:C \rightarrow \bP^1$ the map given by the pencil $U=W$, the zero locus of $Q$ in $S$ is $2R_h  +2F$ where $F =  (B_h \times C) \cup (C \times B_h)$ with $B_h$ the base locus of $h$ and  
$$R_h := \{(p,q) \in S \ | \  h(p) = h(q)\}.$$

Now we give another application. 

\begin{proposition}
\label{rankQ'}
If $Q$ is non zero, then the rank of $Q'$ is $\geq 2$. 
If  rank of $Q=3 $ then the rank of $Q'$ is   $\geq 3.$
\end{proposition}

\begin{proof} First we show that the rank of $Q'$ is $>1$ for any non zero $Q.$  In fact if $Q'=X^2$ we get 
$$Z(Q')=\{(P,Q)\in S \ | \  X(Q)X(P)=0\}=\{(P,Q) \ | \ X(P)=0\}\cup \{(P,Q) \  | \  X(Q)=0\}, $$ that is a union of coordinate curves of the type
$\{P\}\times C \cup C\times \{P\}$, so all the components meet  the diagonal. On the other hand,  we get that
$Z(Q')=Z(\hat{\eta})\cup Z(Q)$ but $Z(\hat{\eta})$ does not intersect the diagonal, as one can easily see from the definition of $\hat{\eta}$, hence $Z(\hat{\eta})$ does not contain any coordinate curve.
Now assume $Q'=XY$, so that  $$Z(Q')=\{(P,Q)\in S \ | \  X(Q)Y(P)+X(P)Y(Q)=0\}$$
That is, if $P$ and $Q$ are not in the base locus of the pencil,  
$$\frac{X(Q)}{Y(Q)}=-\frac{X(P)}{Y(P)}.$$

Now if  $f:C\to \bP^1$ is the map defined by   the sections $X$ and $Y$ of $H^0(2K_C)$, we get
$ Z(Q')= \{(P,Q)\in S: f(P)=-f(Q)\}\cup B=A\cup B$ where $B$ is union of coordinate curves defined by the base locus of the pencil. Let $e$ be the degree of $f.$ For general $t\in \bP^1$ then $f^{-1}(t)=\{P_1,\dots,P_e\}$ and
 $f^{-1}(-t)=\{Q_1,\dots,Q_e\}$ are distinct. It follows that $A$ has not multiple components.
 Now assume that the rank of $Q$ is $3$.  In this case we have 
 $Z(Q)= 2R_h +2F$  as above. The equation $ 2R_h+Z(\hat{\eta})+2F= A+B$ gives $2R_h+Z(\hat{\eta})=A$ and then  we get a contradiction, since
$A$ has not multiple components. 
\end{proof}

Theorem \ref{interpretation} gives a geometric interpretation and easier proof of Theorem 4.1 of \cite{cfg}. This was the main ingredient used in \cite{cfg} to give a bound on the the maximal dimension of a germ of a totally geodesic submanifold generically contained in the Torelli locus. We recall the main results of \cite{cfg}. 
\begin{theorem} [\cite{cfg}]

\label{stima1}
  Assume that $C$ is a $k$-gonal curve of genus $g$ with $g\geq 4$ and
  $k\geq 3$. Let $Y$ be a germ of a totally geodesic submanifold of
  ${\mathsf A}_g$ which is contained in the jacobian locus and passes through
  $j([C])$. Then $\dim Y \leq 2g+k - 4$.

\end{theorem}

From this, using that the gonality is at most  $ [(g+3)/{2}]$, they get the following 

\begin{theorem}  [\cite{cfg}]
  \label{stima2}
  If $g\geq 4$ and $Y$ is a germ of a totally geodesic submanifold of
  $A_g$ contained in the jacobian locus, then $\dim Y \leq
  \frac{5}{2}(g-1)$.
\end{theorem}

The strategy used in \cite{cfg} to prove Theorem \ref{stima1} is to construct  a rank 4 quadric $Q \in I_2(K_C)$ such that the quadric $\rho(Q) \in S^2H^0(2K_C) \cong S^2H^1(T_C)^{\vee}$ has rank at least $2g-2k-1$. This was done in \cite[Thm 4.1]{ cfg} and now follows from Theorem \ref{interpretation}. Then the proof of Theorem  \ref{stima1} follows by observing that the tangent space to a totally geodesic submanifold contained in the Torelli locus and passing through $j(C)$ is a linear subspace of $H^1(T_C)$ which is isotropic with respect to $\rho(Q)$, hence its dimension is at most $3g-3 - (2g-2k-1) + \frac{(2g-2k-1)}{2}=
2g+ k - \frac{7}{2}$.

\section{Gonality}
Let $C$ be a curve of genus $g>5$. Assume that $C$ is not hyperelliptic, and denote by $k$ its gonality. We have $2<k<g-1$.
Let $L$ be  a line bundle  on $C$ of degree $k$  and such that $h^0(L)=2.$ Set $M:=K_C-L$, where $K_C$ is the canonical bundle.
From Riemann Roch we get $h^0(M)=2g-2-k+2-(g-1)=g+1-k>2.$
The Clifford index  of $C$ is either $k-2$ (computed by $L$) or $k-3$.
We have the following

\begin{proposition} Let $B$ be the base locus of $M=K_C-L.$ Then
 \begin{enumerate}  \item $B$ is either empty or $B=P$ is a point. 
 \item  If $B=P$ is a point, then  $h^0(L(P))=3,$ the map 
$f: C\to |L(P)|$ is birational onto its image which is a smooth curve.
\item If $B=\emptyset,$  and $D>0$ is a divisor  such that $h^0(M(-D))=h^0(M)-1, $ then   $\deg D\leq 3.$ 
\item If $B=\emptyset,$ and $D, E>0$  are divisors such that $h^0(M(-D))= h^0(M(-E)) = h^0(M)-1,$ $\deg D = \deg(E) =3,$ then $D=E.$
In particular  the map $g: C\to |M|$ is either birational on its image or has degree $ 2.$
\end{enumerate}
 \end{proposition}
\begin{proof} If $B\neq \emptyset$ is the base locus of  $M$, we have that $h^0(M(-B))=h^0(M).$
Then by Riemann Roch,  $h^0(L(B))=2+\deg B$, the Clifford index of $L(B)$ is  $k+\deg B-2(1+\deg B)=k-2-\deg B\geq k-3$,  hence $B=P$ is a point, and $h^0(L(P))=3.$ Consider the map  $f:C\to |L(P)|$.  It is birational onto its image which is smooth. In fact,  if we set $s=\deg f(C)$ and $d = deg(f)$,  we have $sd=k+1$. Considering the composition of $f$ with the projection from a point $q\in f(C)$ of multiplicity $m_q$, we get a map $C\to \bP^1$ of degree
$d(s-m_q)=k+1-dm_ q\leq k$. Then  $dm_q= 1$, $d=1$, that is $f$ birational,  and for any $q\in f(C)$, $m_q=1$,  hence $q$ is a smooth point.

\bigskip Assume now that $B=\emptyset.$ Let $D>0$ be a divisor of degree $s$ such that $h^0(M)=1+h^0(M(-D))$. Then again from Riemann Roch we have $h^0(L(D))=2+s-1=s+1$, 
$\deg(L(D))=k+s$. Hence its  Clifford index is  $k+s-2s=k-s\geq k-3,$  so $s\leq 3.$ 

Assume that there are two divisors
$D$ and $E$ of degree $3$ such that $h^0(M(-D))=h^0(M(-E))=h^0(M)-1.$
If   $F= D\cap  E$ (this is the MCD of the divisors) we would like to show that $F=\emptyset.$ In fact, set $G=D+E-F$  and consider the exact sequence:
$$0\to M(-G)\to  M(-D)\oplus M(-E)\to M(-F) \to 0.$$
If by contradiction we assume $F\neq \emptyset$, we have $h^0(M-E) \leq h^0(M-F)<h^0(M)$ then $ h^0(M-F)=h^0(M)-1.$ It follows that $h^0(M-G)\geq 2(h^0(M)-1)-(h^0(M)-1)=h^0(M)-1$, 
but  $\deg G>3$ and we find a contradiction.  So we have $F=\emptyset$. Consider the exact sequence:
$$0\to L\to  L(D)\oplus L(E)\to L(D+E)\to 0$$
that in cohomology gives
$0\to H^0(L)\to H^0(L(D))\oplus H^0(L(E))\to H^0(L(D+E)).$ Since $h^0(L(D))=h^0(L(E))=4$,  we obtain
 $H^0(L(D+E))\geq 6$ and its Clifford index $\leq k+6-10= k-4< k-3$, so we get a contradiction.

\end{proof}
From this we immediately get the following
\begin{corollary} Let $C$ be a projective curve
of genus $g$, gonality $k=\deg L$ where $h^0(L)=2$. Then if $C$ is not a smooth plane curve and has no involutions, the linear system $|M|=|K_C-L|$ is base point free and $C\to |M|$ is birational onto its image.\end{corollary}
\section{Divisors and  quadrics}

\subsection{The quadrics and the second fundamental form}
Here we follow \cite{cfg}.
Let $C$ be a smooth projective curve, $L$ a line bundle on $C$ and $M=K_C(-L)$. 
In practice we will consider the case where $L$ computes the gonality, but for the moment we only assume 
\begin{enumerate}
\item $h^0(L)\geq 2,$ 
\item $h^0(M)=r+1\geq 3,$   
\item $M$ base point free,
\item $f: C\to |M|=\bP^r $ is birational onto its image.
\end{enumerate}
Set  $\deg M=d$ and  fix now and for all two indipendent  sections $x,y$ of $L$.  Let $\omega=\mu_{1,L}(x\wedge y)\in H^0(K_C(L^2))$ be the image of the first gaussian (or Wahl) map. The zero divisor $Z$ of $\omega$ is $Z=B+2F$ where $B$ is the branch divisor of $h=x/y$ and $F$ is the fixed divisor (that  in the application  will be empty).
For any global section $s$ of $M$ we consider the associated divisor $D(s)\in |M|.$ 
We fix  a general section $t$ 
such that:
\begin{enumerate}
\item $D(t)\cap Z=\emptyset.$
\item $D(t)=p_1+\dots +p_d$, $p_i\neq p_j$ if $i\neq j.$
\item The points $p_i$ are in general linear position: for any group of or $r$ distinct points $p_{i_1},\dots p_{i_j},\dots p_{i_r},$
we have $H^0(M-(p_{i_1}+\dots  +p_{i_r}))=\langle t \rangle$.
\end{enumerate}
The last condition follows for instance  from the uniform lemma of Castelnuovo (see e.g.  \cite[Ch.3]{acgh}) since $f:C\to \bP^r$ is birational onto its image.

Consider the exact sequence induced by $t$
$$0\to \cO_C\stackrel{t}\to M\to M_D\to 0.$$
We  get  $$M_D=\sum M_{p_i}\cong \sum\bC_{p_i},$$
where the last isomorphism follows from the choice of local trivializations of $M.$
Let $W\subset H^0(M)$ be complementary to $t:$ $H^0(M)=\langle t \rangle \oplus W$, so that $\dim W=r$. 
Consider the induced injection $j: W\to H^0(\oplus \bC_{p_i})=\bC^d.$ 
 We can rewrite the linear uniform condition.
For any $s\in W$, $s\neq0$,  then the vector 
$j(s)=(a_1,\dots,a_i,\dots a_d)$
has at most $r-1$ coordinates that are zero.

Let $I_2(K_C)\subset S^2 H^0(K_C)$  be the vector space of quadrics that contain the canonical image of $C$. 
For any $s\in W$,  set  $\omega_1=xt ,\omega_2=yt, \omega_3=xs$ and $\omega_4=ys$. We define the rank $4$ quadric $Q_s\in I_2$ by
$$Q_s=\omega_1\otimes \omega_4-\omega_2\otimes \omega_3.$$ 
Denote by $\mu_{1,M}: \Lambda^2H^0(M) \rightarrow H^0(K_C\otimes M^2)$ the first Gaussian map of $M$. 
Let $V := \langle \xi_{p_1},..., \xi_{p_d}\rangle \subset H^1(T_C)$ be the subspace generated by the Schiffer variations $\xi_{p_i}$ at the points $p_i$. 
To be more precise, we choose a  coordinate $z_i$ around any point $p_i$ and a trivialization $\sigma$ of $L$  such that
$t_i=z_i\sigma.$ We then let $$\xi_i=\delta (\frac{1}{z_i}\partial/\partial z_i)$$ under the coboundary map $\delta:H^0(T(p_i))\to H^1(T_C).$
Dually we have a surjection $H^0(K^2_C)\to V^\ast.$ 


We  also define 
\begin{equation}
\label{pi}
\pi: W\to S^2V^\ast, \ \pi(s)(v \odot w) = \rho(Q_s)(v \odot w), \ \forall v,w \in V.
\end{equation}

The following fundamental result has been proved  in  
 \cite[Thm. 4.1]{ cfg}, with the variant here that all the quadrics $Q_s,$ $s\in W,$ are taken into account.
It also immediately follows from Theorem \ref{interpretation}. 
\begin{theorem}
\label{diago}
Let $x_1,\dots x_d$ be the basis of $V^\ast$ dual to the basis $\{\xi_{p_1},...,\xi_{p_d}\}$ of $V$ given by the Shiffer variations.
Then $$\pi(s)= \lambda \sum_{i=1}^d a_ix_i^2$$ where $\lambda\neq 0$ is a constant independent of $s$,  
$j(s)=(a_1,\dots,a_d)$ where $ j:W\to \bC^d$ is the evaluation map.
The quadrics $\pi(s)$ are diagonalized  at the same time and  for any  $s\neq 0,$ 
${\rm{rank}} (\pi(s))\geq d-r+1.$
\end{theorem}
\begin{proof} The fact that all the quadrics are in diagonal form follows from  \cite[Thm. 4.1]{ cfg}, or  from Theorem \ref{interpretation}. 
Recall formula \eqref{rhomu_equazione}: $\rho(Q_s)(\xi_{p_i} \odot \xi_{p_i}) = -2 \pi i \mu_2(Q_s) (u_i^4) = -2 \pi i \mu_{1,L}(x\wedge y)\cdot \mu_{1,M}(t\wedge s)(u_i^4)$ where $u_i=\partial/\partial z_i$ and $z_i$ is the coordinate centred at $p_i$. 
Since $\mu_{1,L}(x\wedge y)$ does not depend on $s$, we have to evaluate $ \mu_{1,M}(t\wedge s)$. Now  $ \mu_{1,M}(t\wedge s)= (t's-s't)(p_i)=(t's)(p_i)$  since $t(p_i)=0$ by construction.  Then
 $a_i$ is the evaluation of $s$ at $p_i$. \end{proof}
\subsection{Le zero locus of the quadrics}
For an element $z \in V$, write $z = \sum_{i=1}^d z_i \xi_{p_i}$ and denote by $[z]:=[z_1,...,z_d] \in  \bP^{d-1} \cong  \bP(V)$.
Consider the locus
$$Z=\{[z] \in \bP^{d-1}: \pi(s)(z \odot z)=\rho(Q_s)(z \odot z)=0,  \ \forall s\in W\}.$$ 
 First we have the following
\begin{lemma} Set $H=\{[z] \in \bP^{d-1}: z_i=0,\ i>r\},$ then $Z \cap H=\emptyset,$  therefore  $\dim Z=  d-r-1.$ 
 \end{lemma}
\begin{proof} 
Notice that by the uniform position (see e.g.  \cite[Ch.3]{acgh}), we know that for all $i \in \{1,...,r\}$, there is exactly  a section $s_i\in W$   such that  $s_i(p_j) =0$, $\forall j \in \{1,...,r\}$, $j \neq i$, $s_i(p_i) \neq 0$, $s_i(p_k) \neq 0$, $\forall k >r$, hence by Theorem \ref{diago} we get 
\begin{equation}
\label{bau}
\pi(s_i)=  a_{i,i}x^2_i+\sum_{j>r} a_{i,j}x^2_j,
\end{equation}
  with $a_{i,i} \neq 0$, $a_{i,j} \neq 0$, $\forall j >r$. 
Take $[z] = [z_1,...,z_r,0,...,0]\in H$ such that $\pi(s)(z \odot z) =\rho(Q_s)(z \odot z)=0$, $\forall s \in W$.
Set $Q_i:= Q_{s_i},$ then $\rho(Q_i)(z \odot z)= a_{i,i} z_i^2 =0$ $\forall i$ if and only if  $z_i=0$, $\forall i =1,...,r$, which is impossible, since $[z] \in H$.  Therefore we have  $H\cap Z=\emptyset$
and then $\dim Z \leq d-1-r.$ Notice that $Z=\{[v]\in \bP^{d-1}: \rho(Q_i)(v \odot v)=0, \ i=1,\dots,r\}$, so $\dim Z = d-1-r$.


\end{proof}

\subsection {Estimate}
We need to estimate the dimension of a linear space $\Pi\subset Z.$ Denote by $T$ the linear subspace of $V$ corresponding to $\Pi$. 

Consider the maps $\pi_i:V\to V$,  $\pi_i (x_1,\dots x_i, x_{i+1},\dots,x_d)= (0,\dots,0,x_{i+1},\dots x_d)$ 
The restriction of $\pi_i$ to $T$ is injective for $i \leq r$ since $\Pi\subset Z$. 
By formula \eqref{bau} we can see $\pi(s_r)$ as a quadric in $\pi_{r-1}(V)$. 

We have the inclusion $$\pi_{r-1}(T)\subset \{v \in \pi_{r-1}(V) \ | \ \pi(s_r)(v \odot v) = \rho(Q_{r})(v \odot v)=0\}.$$ Since
$\pi(s_{r})$ has rank $d-r+1$, $\dim(\pi_{r-1}(V)) = d-r+1$ and $\pi_{r-1}: T \la V$ is injective,  we get:
\begin{proposition}
\label{bound}
With the previous notation, let $\Pi$ be a linear subspace contained in $Z,$ and let $T$ be the corresponding  subspace  of $V$.  Then 
$$\dim T \leq  \frac{d-r+1}{2}.$$ 
\end{proposition}

\section{Application}
We assume that $C$ is a curve of genus $g>5,$ gonality $k$ computed by $L$ and assume that $C$ has no involutions and is not a smooth plane curve.
Then we can apply to the general section of $M=K_C-L$ the estimate of the previous section. We have:
$\deg M= 2g-2-k$ and $h^0(M)=g+1-k$,  that is $r=g-k.$   
Then if $t$ is a general section  of $M$ and $V$ is generated by the Schiffer variations
at the zeroes of $t$, by Proposition \ref{bound} we get that a linear space $T\subset V$  contained in the zero locus of the quadrics $\rho(Q_s)$, for $s \in W$  has dimension
\begin{equation}
\label{mao}
n\leq \frac {g-1}{2}.
\end{equation}
 Then we obtain the following
\begin{theorem}
\label{gonality}
If $C$ is a smooth curve of genus $g>5$, gonality $k$, it has no involutions and is not a smooth plane curve,  then any totally geodesic subvariety generically contained in the Torelli locus and passing through $j(C)$ has
dimension $$m\leq  \frac {3(g-1)}{2}+k.$$
\end{theorem}
\begin{proof} Let $S$ be the  tangent space at $[C]$ of a totally geodesic subvariety. Let $V$ be as above,  then $T=S\cap V$ is a linear subspace where all
the quadrics vanish, so by \eqref{mao} we get:  $\dim (S\cap V)\leq \frac{g-1}{2}$.  Then $\dim S+ \dim V \leq 3g-3 +\dim (S\cap V)$, hence 
$$\dim S\leq 3g-3 -(2g-k-2)+\frac{g-1}{2}= \frac {3(g-1)}{2}+k.$$ 
\end{proof}
\begin{theorem}
Let $Y$ be a germ of a totally geodesic submanifold generically contained in the Torelli locus $\Tg$, $g>5$,  then  
$\dim Y \leq 2g-1$ if $g$ is even, $\dim Y \leq 2g$ if $g$ is odd. 
\end{theorem}
\begin{proof} 
For $g>5$ the result follows by Theorem \ref{gonality}, recalling that $k\leq [(g+3)/2].$ For $g=4,5$ the result follows from \cite[Thm.4.4]{cfg}. 
\end{proof}



%
\section{The hyperelliptic locus}

Assume that $C$ is a hyperelliptic curve of genus $g \geq 3$. Denote by $L$ the line bundle giving the $g^1_2$, $H^0(L) = \langle x,y\rangle$. Set $M = K_C(-L)$, denote by $\pi:C \rightarrow \bP^1$ the map induced by $|L|$. Call $\nu_n: \bP^1 \hookrightarrow \bP^{n}$ the $n$th Veronese embedding.  The canonical map is the composition $\nu_{g-1} \circ \pi$, so $K_C \cong \pi^*({\mathcal O}_{{\bP}^1}(g-1)) \cong L^{\otimes (g-1)}$ and $M \cong  L^{\otimes (g-2)}$. Then $H^0(C,M) \cong H^0(\bP^1, {\mathcal O}_{{\bP}^1}(g-2))$  has dimension $r+1= g-1$, so $r = g-2$.  Denote by $\sigma$ the hyperelliptic involution and write   $H^0(C, 2K_C) \cong H^0(C, 2K_C)^+ \oplus  H^0(C, 2K_C)^-$ the decomposition of  $H^0(C, 2K_C)$ in invariant and anti-invariant subspaces by the action of $\sigma$.  By the projection formula one gets: $H^0(C, 2K_C)^+ \cong H^0(\bP^1, {\mathcal O}_{{\bP}^1}(2g-2))$, $H^0(C, 2K_C)^- \cong H^0(\bP^1, {\mathcal O}_{{\bP}^1}(g-3))$. 
	Denote by ${\mathsf {HE}}_g$ the hyperelliptic locus in ${\mathsf M}_g$ and by $j_{h}: {\mathsf {HE}}_g \rightarrow {\mathsf A}_g$ the restriction of the Torelli map to ${\mathsf {HE}_g}$. Then $j_h$ is an orbifold immersion and we have the following tangent bundle exact sequence


\begin{gather}
\label{tangent}
  \xymatrix{& & 0 \ar[d] & & &\\
    &0\ar[r] &  T_{{\mathsf {HE}}_g}\ar[d]\ar[r] &  {T_{{\mathsf A}_g}}_{|{\mathsf{HE}_g}}\ar[d]^=\ar[r] & N_{{\mathsf {HE}}_g/{\mathsf A}_g} \ar[r]\ar[d] & 0\\
    &  &{T_{{\mathsf M}_g}}_{|{\mathsf HE}_g} \ar[r]&  {T_{{\mathsf A}_g}}_{|{\mathsf{HE}_g}}\ar[r]& {N_{{\mathsf {M}}_g/{\mathsf A}_g}}_{| {{\mathsf {HE}}_g}}
     \ar[r] &0\\
    & & & &  &}
\end{gather}

Denote by 
\begin{equation}
\rho_{HE}: N^*_{{\mathsf {HE}}_g|{\mathsf A}_g} \rightarrow S^2 T_{{\mathsf {HE}}_g}
\end{equation}
the dual of the second fundamental form of $j_h$. 

At a point $[C] \in {\mathsf {HE}}_g$, the dual of \eqref{tangent} is

\begin{gather}
\label{cotangent}
  \xymatrix{
        &0 \ar[r] & I_2(K_{C}) \ar[d]^=\ar[r]& S^2 H^0(K_{C})\ar[d]^=\ar[r]^m& H^0(2K_{C})\ar[d] &\\
    &0\ar[r] & I_2(K_{C})\ar[r] &  S^2 H^0(K_{C})\ar[r]^m &  H^0(2K_{C})^+\ar[r] & 0}
\end{gather}


and $I_2(K_C)$ can be identified with the set of quadrics containing the rational normal curve.  

We have $\forall Q \in I_2(K_C)$, $\forall v,w \in H^1(T_C)^+$, $\rho_{HE}(Q) (v \odot w) = \rho(Q)(v \odot w)$ (see \cite[Prop. 5.1]{ cf1}).  Here  we denote by $\rho(Q)$ the section $ Q \cdot \hat{\eta} $, seen as an element in $S^2 H^0(K_C)$ as in Theorem  \ref{prodotto}. 

Take $t \in H^0(M) $ a generic section and denote by $D(t)$ its divisor. It has degree $2g-4$ and  it is invariant by the action of $\sigma$, hence $D(t) = q_1 + \sigma(q_1)+...+q_{g-2} + \sigma(q_{g-2})$. Write as above $H^0(M) = \langle t \rangle \oplus W$. Since $H^0(C,M) \cong H^0(\bP ^1, {\mathcal O}_{\bP^1}(g-2))$, to the section $t \in H^0(C, M)$ corresponds a section $\overline{t} \in  H^0(\bP ^1, {\mathcal O}_{\bP^1}(g-2))$, whose zero divisor is $\overline{D} = r_1 + ...+r_{g-2}$ where $r_i = \pi(q_i)$, $\forall i$.  Denote by $\overline{W}$ the isomorphic image of $W$ in $H^0(\bP ^1, {\mathcal O}_{\bP^1}(g-2))$. For every $i \in \{1,...,g-2\}$ there exists a unique section $\bar{s_i} \in \overline{W}$ such that $\bar{s_i}(r_j) =0 $ for all $j \neq i$ and $\bar{s_i}(r_i) \neq 0$. The section $\bar{s_i}$ coresponds to a section $s_i \in W$ such that $s_i(q_j) = 0$ for all $j \neq i$ and $s_i(q_i) \neq 0$. 

For every section $s \in W$ consider as above the quadric 
$$Q_s = (xt) \odot (ys) - (xs) \odot (yt) \in I_2(K_C),$$
and the tangent vectors $v_i := \xi_{q_i} + \xi_{\sigma(q_i)}$, $i =1,...,g-2$. Clearly $v_i \in H^1(T_C)^+$, hence by the $\sigma$-invariance, and using \eqref{rhomu_equazione},  if $i \neq j$ we have
$$\rho_{HE}(Q_s)(v_i \odot v_j) = \rho(Q_s)(v_i \odot v_j)=   2\rho(Q_s)(\xi_{q_i} \odot \xi_{q_j}) +  2\rho(Q_s)(\xi_{q_i} \odot \xi_{\sigma(q_j)}) = $$
$$=8 \pi i (Q_s(q_i, q_j) \eta_{q_i}(q_j) + Q_s(q_i, \sigma(q_j) )\eta_{q_i}(\sigma(q_j))  ) =0,$$
since $q_i, q_j$ and $\sigma(q_j)$ are in the zero locus of $t$. 
On the other hand we have 
  $$\rho_{HE}(Q_s)(v_i \odot v_i) = \rho(Q_s)(v_i \odot v_i)=   2\rho(Q_s)(\xi_{q_i} \odot \xi_{q_i}) +  2\rho(Q_s)(\xi_{q_i} \odot \xi_{\sigma(q_i)}) = $$
$$=8 \pi i (2\mu_2(Q_s)(q_i) + Q_s(q_i, \sigma(q_j)) \eta_{q_i}(\sigma(q_j))  ) =16 \pi i \mu_2(Q_s)(q_i) = $$
$$=16 \pi i \mu_{1,L}(x \wedge y)(q_i) \mu_{1, M}(s \wedge t)(q_i) = 16 \pi i a_i,$$
where  $a_i =0$ if and only if $s(q_i) =0$. 

So if we denote by $V = \langle v_1,...,v_{g-2}\rangle$,  we have shown that $\rho_{HE}(Q_s)_{|V}$ is diagonal with diagonal entries equal to $ca_i$, $c = 16 \pi i$. 
\begin{proposition}
\label{HE}
Assume that $T \subset H^1(T_C)^+$ is a linear subspace which is isotropic with respect to all the quadrics $\rho_{HE}(Q_{s})$, $\forall s \in W$, then $T \cap V = \{0\}$. 

\end{proposition}
\proof 
If $v = z_1 v_1 + ...+ z_{g-2} v_{g-2} \in V$, then $\rho_{HE}(Q_{s_i})(v \odot v) = c a_i z_i^2$, with $a_i \neq 0$. Hence $\rho_{HE}(Q_{s_i})(v \odot v) =0$ if and only if $z_i =0$. 
\qed
\begin{theorem}
Let $Y$ be a germ of a totally geodesic submanifold of ${\mathsf A}_g$ contained in the hyperelliptic locus, then $\dim Y \leq g+1$. 
\end{theorem}
\proof
If $Y$ is a germ of a totally geodesic submanifold of  ${\mathsf {HE}}_g$ passing to $[C]$,  its tangent space at $[C]$ is isotropic with respect to all the quadrics $\rho_{HE}(Q)$ with $Q \in I_2(K_C)$,  hence to all the quadrics  $\rho_{HE}(Q_{s_i})$, therefore $T_{[C]}Y \cap V =\{0\}$ by Proposition \ref{HE}.
Thus $\dim T_{[C]}Y + \dim V = \dim Y + g-2 \leq \dim H^1(T_C)^+ = 2g-1$, so $\dim Y \leq g+1$. 
\qed

\begin{proposition}
Let $Y$ be a germ of a totally geodesic submanifold of ${\mathsf A}_3$,  contained in the hyperelliptic locus, then $\dim Y \leq 3$. 
\end{proposition}
\proof

If $g =3$ then the dimension of the space of quadrics containing the rational normal curve is one and this space is generated by the rational normal curve, which is a smooth conic $Q$. By Proposition \ref{rankQ'}, the rank of $\rho(Q)$ is at least 3, hence $\dim(Y) \leq 5-3+\frac{3}{2} = \frac{7}{2}$. 
\qed

\begin{remark}
Families (8) and (22) of Table 2 of \cite{fgp} (see also Tables 1 and 2 of \cite{moonen-oort}) yield respectively a two-dimensional and a one-dimensional Shimura (hence totally geodesic) subvariety of ${\mathsf A}_3$ generically contained in the hyperelliptic Torelli locus. 

Family (36) of Table 2 of \cite{fgp} yields  a one-dimensional Shimura (hence totally geodesic) subvariety of ${\mathsf A}_4$ generically contained in the hyperelliptic Torelli locus. 

Family (39)  of Table 2 of \cite{fgp} yields  a one-dimensional Shimura (hence totally geodesic) subvariety of ${\mathsf A}_5$ generically contained in the hyperelliptic Torelli locus. 

\end{remark}
\proof
See \cite{fgp} 4.6. 
\qed

\end{document}